\documentclass[12pt,reqno]{amsart}
\pdfoutput=1
\usepackage[margin=1in]{geometry}
\usepackage{amsmath,amssymb,amsthm}

\usepackage{tikz}
\usepackage{enumitem}
\usepackage[T1]{fontenc}

\newcommand{\B}{\mathbf{B}}
\newcommand{\G}{\mathbf{G}}

\newcommand{\X}{\mathbf{X}}
\newcommand{\Q}{\mathbb{Q}}

\newcommand{\Zpos}{\mathbb{Z}_{>0}}
\renewcommand{\d}{\partial}
\newcommand{\til}{\widetilde}
\renewcommand{\l}{\lambda}
\newcommand{\Sym}{\operatorname{Sym}}
\newcommand{\Gk}{\mathbf{\Gamma}_T^{(k)}}
\newcommand{\Ga}[1]{\mathbf{\Gamma}_T^{(#1)}}
\newcommand{\tensor}{\otimes}
\newcommand{\verts}[1]{\overline{#1}}
\newcommand{\nlverts}[1]{\widehat{#1}}

\newtheorem{theorem}{Theorem}[section]
\newtheorem{proposition}[theorem]{Proposition}

\newtheorem{lemma}[theorem]{Lemma}
\theoremstyle{definition}

\newtheorem{example}[theorem]{Example}
\newtheorem{defn}[theorem]{Definition}

\newtheorem{question}[theorem]{Question}
\theoremstyle{remark}

\usepackage[doi=false,isbn=false,url=false,eprint=false,maxnames=50]{biblatex}
\usepackage{url}

\numberwithin{equation}{section}
\title{Generalized degree polynomials of trees}

\author{Ricky Ini Liu}
\address{Department of Mathematics, University of Washington, Seattle, WA 98195}
\email{riliu@uw.edu}

\author{Michael Tang}
\address{Department of Mathematics, University of Washington, Seattle, WA 98195}
\email{mst0@uw.edu}

\date{\today}

\begin{document}

\begin{abstract}
    The \emph{generalized degree polynomial} $\mathbf{G}_T(x,y,z)$ of a tree $T$ is an invariant introduced by Crew that enumerates subsets of vertices by size and number of internal and boundary edges. Aliste-Prieto et al.\ proved that $\mathbf{G}_T$ is determined linearly by the \emph{chromatic symmetric function} $\mathbf{X}_T$, introduced by Stanley. We present several classes of information about $T$ that can be recovered from $\mathbf{G}_T$ and hence also from $\mathbf{X}_T$. Examples of such information include the \emph{double-degree sequence} of $T$, which enumerates edges of $T$ by the pair of degrees of their endpoints, and the \emph{leaf adjacency sequence} of $T$, which enumerates vertices of $T$ by degree and number of adjacent leaves. We also discuss a further generalization of $\mathbf{G}_T$ that enumerates tuples of vertex sets and show that this is also determined by $\mathbf{X}_T$.
\end{abstract}

\maketitle

\section{Introduction}

Given a simple graph $G$, the \emph{chromatic symmetric function} $\X_G$, introduced by Stanley \cite{stanley} in 1995, is a symmetric function that enumerates proper colorings of $G$ by the number of occurrences of each color. It vastly generalizes the classical chromatic polynomial $\chi_G$ of Birkhoff and is a natural graph invariant in that the mapping $G \mapsto \X_G$ defines a canonical morphism from the Hopf algebra of graphs to the ring of quasisymmetric functions \cite{abs}.

Stanley \cite{stanley} asked the following question which motivates our present work:

\begin{question} \label{q:distinguish}
    Does the chromatic symmetric function distinguish trees? In other words, must non-isomorphic trees have different chromatic symmetric functions?
\end{question}

While this question remains open, some progress has been made towards answering it. One common approach to doing so is to identify or construct special classes of trees that are distinguished by their chromatic symmetric functions, as in \cite{aliste-prieto, ganesan, loebl, wang, gonzalez}. Another approach is to begin with the chromatic symmetric function $\X_T$ of an arbitrary tree $T$ and extract information about $T$ from the coefficients of $\X_T$. For instance, Martin, Morin, and Wagner~\cite{mmw} proved that $\X_T$ determines the so-called \emph{subtree polynomial} of $T$ (see e.g.\ \cite{eisenstat}), which in turn determines the \emph{degree sequence} (the multiset of vertex degrees) of $T$. As another example, Crew \cite{crewnote} showed that $\X_T$ determines the multiset of lengths of the \emph{twigs} of $T$, which are the maximal paths in $T$ in which one end is a leaf and all other vertices have degree $2$ in $T$.

In this paper, we mainly adopt the latter approach: recovering information about $T$ from $\X_T$. In particular, we study the \emph{generalized degree polynomial} $\G_T(x,y,z)$ of a tree $T$. This invariant, introduced by Crew \cite{crewthesis}, enumerates subsets $A$ of the vertices of $T$ by size as well as number of internal edges (edges of $T$ with both endpoints in $A$) and boundary edges (edges of $T$ with exactly one endpoint in $A$). Crew conjectured in \cite{crewnote} that $\X_T$ determines $\G_T$, and this was later proven by Aliste-Prieto, Martin, Wagner, and Zamora~\cite{martin2}. Therefore, any information about $T$ determined by $\G_T$ is also information determined by $\X_T$.
In this paper, we present several examples of such information. These examples are largely organized using the concept of \emph{degree embedding numbers} $N_D(T)$, which count the number of subtrees of $T$ isomorphic to a given tree $D$ and with prescribed vertex degrees. For instance, in the case that $D$ is a single edge, we show that $\G_T$ determines, for all positive integers $a$ and $b$, the number of edges of $T$ whose endpoints have degrees $a$ and $b$; these numbers form the \emph{double-degree sequence} of $T$. In the case that $D$ is a star, we show that $\G_T$ determines (for all $d$ and $\ell$) the number of vertices of $T$ with degree $d$ that are adjacent to exactly $\ell$ leaves; these numbers form the \emph{leaf adjacency sequence} of $T$. (The full list of results can be found in Theorem~\ref{thm:degree-embedding-list}.) These degree embedding numbers $N_D(T)$ are extracted algebraically from $\G_T$, and in most cases can be written as polynomials in the coefficients of $\G_T$.

Then, in the final section, we discuss a further generalization of $\G_T$, the \emph{generalized degree polynomial of order $m$} (for positive integers $m$), denoted $\G^{(m)}_T$. This invariant enumerates tuples of $m$ disjoint subsets $A_1, \dots, A_m$ of vertices of $T$ by the sizes of each $A_i$, the number of internal edges of each $A_i$, and the total number of boundary edges among the $A_i$. We show how the polynomials $\G^{(m)}_T$ arise from the \emph{bad chromatic symmetric function} $\B_T$, a generalization of $\X_T$ that enumerates all (not necessarily proper) colorings of $T$ by the number of vertices of each color and the number of monochromatic edges of each color. Furthermore, we show that $\B_T$ is actually equivalent to $\X_T$, from which it follows that $\X_T$ determines each $\G^{(m)}_T$. The proof uses the structure of the Hopf algebra of graphs \cite[Section 12]{SCHMITT1994299} and generalizes an alternate proof we give that $\X_T$ determines $\G_T$.

This paper is organized as follows. In Section~\ref{sec:gdp}, we define the chromatic symmetric function and the generalized degree polynomial and then give an alternate proof that the former determines the latter. In Section~\ref{sec:degree-embeddings}, we define degree embeddings and present a list of degree embedding numbers that are determined by $\G_T$ (Theorem~\ref{thm:degree-embedding-list}). In Section~\ref{sec:degree-embedding-numbers}, we show how to recover these degree embedding numbers from $\G_T$. Specifically, we construct a family of two-variable rational functions $\Gk(x,y)$ that are determined by $\G_T$ (Section~\ref{subsec:Gk}), and then we use the $\Gk$ to extract the desired data about $T$ (Sections~\ref{subsec:double-degrees}--\ref{subsec:leaf-adjacency}). Finally, in Section~\ref{sec:higher-order}, we define and discuss the higher-order generalized degree polynomials $\G^{(m)}_T$.

\section{The chromatic symmetric function and generalized degree polynomial} \label{sec:gdp}

We begin by introducing notation and background information about the chromatic symmetric function and the generalized degree polynomial. For more background on symmetric functions and Hopf algebras, see e.g.\ \cite{grinberg}.

\subsection{Notation and conventions}

Let $\l = (\l_1, \dots, \l_r)$ be an integer partition, written so that $\l_1 \ge \dotsm \ge \l_r>0$. The \emph{length} of $\l$ is $\ell(\l) = r$, and the \emph{size} of $\l$ is $|\l| = \l_1 + \dotsm + \l_r$. If $|\l| = k$ for some nonnegative integer $k$, we write $\l \vdash k$.

Let $\Sym$ be the ring of symmetric functions in countably many variables $\mathbf{x} = (x_1, x_2, \dots)$ over the rational numbers $\Q$. We will primarily use the \emph{power sum basis} $\{p_\l\}$ of $\Sym$, indexed by integer partitions $\l = (\l_1, \dots, \l_r)$, defined by $p_\varnothing = 1$ for the empty partition $\varnothing$ and \[
    p_\l = p_{\l_1} \dotsm p_{\l_r}
    \qquad \text{where} \qquad
    p_k = x_1^k + x_2^k + \dotsb
\] for positive integers $k$. With the usual grading $\Sym = \bigoplus_{k \ge 0} \Sym_k$ by degree, each power sum symmetric function $p_\l$ is homogeneous of degree $|\l|$, and $\{p_\l : \l \vdash k\}$ is a basis for $\Sym_k$ as a vector space over $\Q$.

In this paper, all graphs are assumed to be simple, that is, they have no parallel edges or self-loops. For a graph $G$, we write $V(G)$ and $E(G)$ for its vertex set and edge set, respectively. The \emph{order} of $G$, denoted $|G|$, is its number of vertices. Let $c(G)$ denote the number of connected components of $G$. We write $\deg_G(v)$ for the degree of a vertex $v \in V(G)$, or just $\deg(v)$ if the context is clear. Finally, for $A \subseteq V(G)$, we write $G|A$ to mean the induced subgraph of $G$ whose vertex set is $A$.

\subsection{The chromatic symmetric function}

For a graph $G$, the \emph{chromatic symmetric function} $\X_G$ is defined by
\[
    \X_G = \sum_{\kappa} \prod_{v \in V(G)} x_{\kappa(v)},
\]
where $\kappa \colon V(G) \to \Zpos$ ranges over proper colorings of $G$ with positive integers. Observe that $\X_G$ is a homogeneous symmetric function of degree $|G|$.

Recall that $\Sym$ has the structure of a Hopf algebra over $\Q$; in particular, it is equipped with a \emph{comultiplication map} $\Delta \colon \Sym \to \Sym \tensor \Sym$. (This can be characterized as the unique linear and multiplicative map that satisfies $\Delta(p_n) = 1 \tensor p_n + p_n \tensor 1$ for $n \in \Zpos.$) Furthermore, for any graph $G$, we have \begin{align} \label{eq:CSF-comult}
    \Delta(\X_G) = \sum_{A \sqcup B = V(G)} \X_{G|A} \tensor \X_{G|B},
\end{align} where $A$ and $B$ form a partition of $V(G)$. This can be proven directly from the definitions or by using the theory of \emph{combinatorial Hopf algebras} \cite{abs}. In particular, this means that if $f$ and $g$ are $\Q$-linear maps $\Sym \to E$ for some $\Q$-algebra $E$, then their \emph{convolution} $f * g$ satisfies \begin{align} \label{eq:CSF-Hopf-map}
    (f*g)(\X_G) = \sum_{A \sqcup B = V(G)} f(\X_{G|A}) \cdot g(\X_{G|B}),
\end{align} and this is also a $\Q$-linear map $\Sym \to E$. Moreover, if $f$ and $g$ are $\Q$-algebra maps, then so is $f * g$, provided that $E$ is commutative.

The definition of the chromatic symmetric function immediately gives a formula for its expansion in the monomial basis of $\Sym$. However, in the case of a forest $F$, the expansion of $\X_F$ in the power sum basis will prove more useful because it has the following combinatorial interpretation.
A \emph{connected partition} of $F$ is an (unordered) set partition of $V(F)$ into blocks, each of which is the vertex set of a connected subgraph of $F$. If $\mathcal{C}$ is a connected partition of $F$, the \emph{type} of $\mathcal{C}$ is the integer partition $\l(\mathcal{C}) \vdash |F|$ whose parts are the sizes of the blocks of $\mathcal{C}$. For an integer partition $\l \vdash |F|$, let $b_\l(F)$ be the number of connected partitions of $F$ of type $\l$. Then we have the following result of Stanley.

\begin{proposition}[{\cite[Corollary 2.8]{stanley}}] \label{prop:CSF-powersum}
    If $F$ is a forest of order $n$, then \[
        \X_F = \sum_{\l \vdash n} b_\l(F) (-1)^{n-\ell(\l)} p_\l.
    \]
\end{proposition}

Thus, for a forest $F$, knowing $\X_F$ is equivalent to knowing all of the numbers $b_\l(F)$.

\subsection{The generalized degree polynomial} \label{subsec:GDP}

We now define the generalized degree polynomial of a forest $F$, following \cite[Section 2.2]{martin2}. Let $A \subseteq V(F)$. An \emph{internal edge} of $A$ is an edge of $F$ with both endpoints in $A$, and a \emph{boundary edge} of $A$ is an edge of $F$ with exactly one endpoint in $A$. (Note that an internal edge of $A$ is the same as an edge of $F|A$.) Let $E_F(A)$ and $D_F(A)$ denote the sets of internal edges and boundary edges of $A$, respectively; we write $e_F(A) = |E_F(A)|$ and $d_F(A) = |D_F(A)|$. (When there is no risk of confusion, we will omit the subscripts.) Observe that \begin{align}
    d_F(A) + 2e_F(A) = \sum_{v \in A} \deg_F(v) \label{eq:deg-identity}
\end{align} by a double-counting argument. Finally, for nonnegative integers $a, b, c$, let \[
    g_F(a,b,c) = \# \{ A \subseteq V(F) : \; |A| = a, \; d(A) = b, \; e(A) = c\}.
\]

\begin{defn}
    The \textbf{generalized degree polynomial} of $F$ is defined by \[
        \G_F(x,y,z) = \sum_{A \subseteq V(F)} x^{|A|} y^{d(A)} z^{e(A)} = \sum_{a,b,c} g_F(a,b,c) x^a y^b z^c.
    \]
\end{defn}

Thus the generalized degree polynomial enumerates subsets $A \subseteq V(F)$ by size and the number of internal and boundary edges. In particular, if $|A| = 1$, then $A$ consists of a single vertex $v$ such that $d(A) = \deg(v)$ and $e(A) = 0$. This implies that $\G_F$ determines the \emph{degree sequence} of $F$: specifically, the number of vertices of $F$ with degree $d$ is just $g_F(1, d, 0)$.

\begin{example}
    Suppose $F$ is a star with $4$ vertices. (Recall that a \emph{star} is a tree consisting of a center vertex connected to some number of outer vertices with no other edges.) Then \[
        \G_F(x,y,z) = x^4z^3 + x^3y^3 + 3x^3yz^2 + 3x^2y^2z + 3x^2y^2 + xy^3 + 3xy + 1.
    \] For example, the coefficient of $x^3yz^2$ in $\G_F(x,y,z)$ is $g_F(3,1,2) = 3$ because there are three subsets of vertices $A \subseteq V(F)$ such that $|A| = 3$, $d(A) = 1$, and $e(A) = 2$. (Specifically, these subsets consist of the center vertex of $F$ and any two of its outer vertices.)
\end{example}

Aliste-Prieto et al.\ \cite{martin2} showed that for a tree $T$, the generalized degree polynomial $\G_T$ is determined linearly by the chromatic symmetric function $\X_T$, that is, they exhibit an explicit $\Q$-linear map $\Sym \to \Q(x,y,z)$ that sends $\X_T$ to $\G_T$ for every tree $T$. We now prove a slight generalization of this statement by using the Hopf algebra structure of $\Sym$.

\begin{proposition} \label{prop:GT-linear}
    There exists a $\Q$-algebra map $\gamma \colon \Sym \to \Q(x,y,z)$ such that for every forest $F$, we have \[
        \gamma(\X_F) = y^{c(F)} \G_F(x,y,z).
    \]
\end{proposition}

To prove this statement, we will use the following lemma.

\begin{lemma} \label{lemma:phi}
    There exists a $\Q$-algebra map $\varphi_{t,u} \colon \Sym \to \Q[t,u]$ such that for every forest $F$, we have \begin{align} \label{eq:phi_tu}
        \varphi_{t,u}(\X_F) = t^{|F|} u^{e(F)}.
    \end{align}
\end{lemma}

\begin{proof}
    We claim that we can define $\varphi_{t,u}$ on the power sums $p_k$ by $\varphi_{t,u}(p_k) = t^k (1-u)^{k-1}$ (for $k \geq 1$) and hence on the power sum basis $\{p_\l\}$ by \[
        \varphi_{t,u}(p_\l) = t^{|\l|} (1-u)^{|\l|-\ell(\l)}. \label{eq:def-phi-powersum}
    \] To show \eqref{eq:phi_tu}, let $F$ be a forest and let $n = |F|$. By Proposition~\ref{prop:CSF-powersum} we have \begin{align}
        \varphi_{t,u}(\X_F)
        &= \sum_{\l \vdash n} b_\l(F) (-1)^{n-\ell(\l)} \cdot t^n (1-u)^{n-\ell(\l)} \nonumber
        \\ &= t^n \sum_{\l \vdash n} b_\l(F) (u-1)^{n-\ell(\l)}. \label{eq:mid-phi-proof}
    \end{align} 
    Note that connected partitions of $F$ are in bijection with subsets of $E(F)$, where $S \subseteq E(F)$ corresponds to the connected partition $\mathcal C$ giving the components of the spanning subgraph of $F$ with edge set $S$. If the type of $\mathcal C$ is $\lambda \vdash n$, then we have $|S| = n-\ell(\lambda)$, so it follows that $\sum_{\ell(\lambda)=i} b_\lambda(F) = \binom{e(F)}{n-i}$. Therefore, letting $i = \ell(\l)$ in \eqref{eq:mid-phi-proof}, we have \[
        \varphi_{t,u}(\X_F) = t^n \sum_i \binom{e(F)}{n - i} (u-1)^{n-i} = t^n u^{e(F)},
    \]
    as claimed.
\end{proof}

We now use Lemma~\ref{lemma:phi} together with \eqref{eq:CSF-Hopf-map} to prove Proposition~\ref{prop:GT-linear}.

\begin{proof}[Proof of Proposition~\ref{prop:GT-linear}]
     Take $\gamma = \varphi_{xy, y^{-1}z} * \varphi_{y, y^{-1}}$. Given a forest $F$, we compute: \[\begin{aligned}
        (\varphi_{xy, y^{-1}z} * \varphi_{y, y^{-1}}) (\X_F)
        &= \sum_{A \sqcup B = V(F)} \varphi_{xy, y^{-1}z}(\X_{F|A}) \cdot \varphi_{y, y^{-1}}(\X_{F|B})
        \\ &= \sum_{A \sqcup B = V(F)} (xy)^{|A|} (y^{-1}z)^{e(A)} \cdot y^{|B|} (y^{-1})^{e(B)}
        \\ &= \sum_{A \sqcup B = V(F)} x^{|A|} \, y^{|F| - e(A) - e(B)} \, z^{e(A)}.
    \end{aligned}\] Each edge of $F$ is either an internal edge of $A$, an internal edge of $B$, or a boundary edge of $A$ (and these are mutually exclusive), so $e(A) + e(B) + d(A) = |E(F)| = |F| - c(F)$. Thus, \[
        (\varphi_{xy, y^{-1}z} * \varphi_{y, y^{-1}}) (\X_F)
        = \sum_{A \subseteq V(F)} x^{|A|} \, y^{c(F)+d(A)} \, z^{e(A)} = y^{c(F)} \G_F(x,y,z)
    \] as desired.
\end{proof}

A couple of remarks are in order regarding the map $\gamma$ constructed above. 
First, since $\gamma$ is multiplicative, it is determined by the values of $\gamma(p_n)$ for $n \ge 1$. Since the comultiplication of $\Sym$ obeys $\Delta(p_n) = p_n \otimes 1 + 1 \otimes p_n$, we have \begin{align*}
    \gamma(p_n)
    &= (\varphi_{xy,y^{-1}z} * \varphi_{y,y^{-1}})(p_n) \\
    &= \varphi_{xy,y^{-1}z}(p_n) \cdot \varphi_{y,y^{-1}}(1) + \varphi_{xy,y^{-1}z}(1) \cdot \varphi_{y,y^{-1}}(p_n) \\
    &= (xy)^n (1-y^{-1}z)^{n-1} + y^n (1-y^{-1})^{n-1} \\
    &= x^n y (y-z)^{n-1} + y (y-1)^{n-1},
\end{align*}
which provides an alternative characterization of $\gamma$. Second, it can be shown by a direct (albeit tedious) calculation that $\gamma$ agrees, up to scaling, with the $\Q$-linear map from \cite{martin2} that sends $\X_T$ to $\G_T$ for trees $T$. Specifically, in their notation, one has \[
    \gamma(p_\lambda) = y \cdot \sum_{a, b, c} \omega(\lambda, a, b, c) x^a y^b z^c.
\]
In Section~\ref{sec:higher-order}, we will discuss a further generalization of the generalized degree polynomial that is obtained by a similar construction.

\subsection{Trees with equal generalized degree polynomials}

By contrast with Question~\ref{q:distinguish}, the generalized degree polynomial certainly does not distinguish trees. The smallest pair of non-isomorphic trees with the same generalized degree polynomial is shown in Figure~\ref{fig:same-gdp}. In fact, this example (or in general any such example) gives rise to infinitely many pairs of trees with the same generalized degree polynomial in the following way.

\begin{figure}
    \centering
    \begin{tikzpicture}
        \draw (0,0)--(1,0)--(2,0)--(3,0)--(4,0);
        \draw (0,0)--(0,-1);
        \draw (1,0)--(1,-1);
        \draw (3,0)--(2.8,-1);
        \draw (3,0)--(3.2,-1);
        \draw (4,0)--(3.8,-1);
        \draw (4,0)--(4.2,-1);
        \filldraw [black] (0,0) circle (3pt);
        \filldraw [black] (1,0) circle (3pt);
        \filldraw [black] (2,0) circle (3pt);
        \filldraw [black] (3,0) circle (3pt);
        \filldraw [black] (4,0) circle (3pt);
        \filldraw [black] (0,-1) circle (3pt);
        \filldraw [black] (1,-1) circle (3pt);
        \filldraw [black] (2.8,-1) circle (3pt);
        \filldraw [black] (3.2,-1) circle (3pt);
        \filldraw [black] (3.8,-1) circle (3pt);
        \filldraw [black] (4.2,-1) circle (3pt);
        
        \draw (6,0)--(7,0)--(8,0)--(9,0)--(10,0);
        \draw (6,0)--(6,-1);
        \draw (7,0)--(6.8,-1);
        \draw (7,0)--(7.2,-1);
        \draw (8,0)--(8,-1);
        \draw (10,0)--(9.8,-1);
        \draw (10,0)--(10.2,-1);
        \filldraw [black] (6,0) circle (3pt);
        \filldraw [black] (7,0) circle (3pt);
        \filldraw [black] (8,0) circle (3pt);
        \filldraw [black] (9,0) circle (3pt);
        \filldraw [black] (10,0) circle (3pt);
        \filldraw [black] (6,-1) circle (3pt);
        \filldraw [black] (6.8,-1) circle (3pt);
        \filldraw [black] (7.2,-1) circle (3pt);
        \filldraw [black] (8,-1) circle (3pt);
        \filldraw [black] (9.8,-1) circle (3pt);
        \filldraw [black] (10.2,-1) circle (3pt);
    \end{tikzpicture}
    \caption{The smallest pair of trees with the same generalized degree polynomial.}
    \label{fig:same-gdp}
\end{figure}
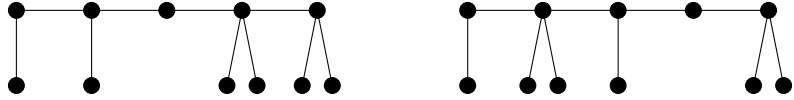

Let $T$ be a tree and let $k$ be a positive integer such that $k \ge \deg(v)$ for every $v \in V(T)$. We define the \emph{$k$-augmentation} of $T$, denoted $T^{[k]}$, to be the tree created by attaching $k - \deg(v)$ leaves to every vertex $v$ of $T$. Thus every non-leaf vertex of $T^{[k]}$ has degree $k$, and the interior of $T^{[k]}$ (the subtree consisting of the non-leaf vertices) is isomorphic to $T$. We therefore view $T$ as embedded in $T^{[k]}$.

\begin{proposition} \label{prop:gdp-augmentation}
    Fix integers $n, k \ge 2$. There exists a $\Q$-linear map $\Q(x,y,z) \to \Q(x,y,z)$ that sends $\G_T$ to $\G_{T^{[k]}}$ for trees $T$ of order $n$.
\end{proposition}

\begin{proof}
    For convenience, let $U = T^{[k]}$. We show that each coefficient of $\G_U$ is a linear combination of coefficients of $\G_T$. Let $a, b, c \ge 0$ and consider \[
        g_U(a,b,c) = \# \{ 
            A \subseteq V(U) : \; |A| = a, \; d_U(A) = b, \; e_U(A) = c
        \}.
    \] By \eqref{eq:deg-identity}, we have $b + 2c = d_U(A) + 2e_U(A) = \sum_{v \in A} \deg_U(v)$. But every vertex of $U$ has degree either $1$ (if it is a leaf of $U$) or $k$ (if it is not a leaf of $U$ and hence a vertex of $T$). Let $\ell_U(A)$ and $i_U(A)$ denote the number of vertices in $A$ of each type, respectively. Then, \[
        \left\{ \begin{aligned}
            \ell_U(A) + i_U(A) &= a, \\
            \ell_U(A) + k \cdot i_U(A) &= b + 2c,
        \end{aligned} \right.
    \] which has a unique solution $(\ell_U(A), i_U(A)) = (\ell, i)$ in terms of $a$, $b$, $c$. It follows that \[
        g_{U}(a,b,c) = \# \{ 
            A \subseteq V(U) : \; \ell_U(A) = \ell, \; i_U(A) = i, \; e_U(A) = c
        \}.
    \] Now let $A_0 = A \cap V(T)$. This set must have size $i$, and it has some number of internal and boundary edges, say $e_T(A_0) = e$ and $d_T(A_0) = d$. (See Figure~\ref{fig:gdp-augmentation-proof} for an example.) We compute the number of choices for $A \setminus A_0$ (consisting of leaves of $U$) such that $\ell_{U}(A) = \ell$ and $e_{U}(A) = c$. First, observe that the number of leaves of $U$ is \[
        \ell^* = \sum_{v \in V(T)} (k - \deg_T(v)) = kn - (2n-2),
    \] which is a constant. Of those leaves, the number that are adjacent to a vertex in $A_0$ is \[
        \sum_{v \in A_0} (k - \deg_T(v)) = ki - (d + 2e),
    \] again using \eqref{eq:deg-identity}. Now $A \setminus A_0$  must consist of $\ell$ of these leaves, and in order to make $e_U(A) = c$, exactly $c - e$ of the leaves must be adjacent to vertices in $A_0$. Therefore we have \[
        g_U(a,b,c)
        = \sum_{d, \, e} g_T(i, d, e) \binom{ki-(d+2e)}{c-e} \binom{\ell^* - (ki-(d+2e))}{\ell - (c-e)}.
    \] Hence $g_U(a,b,c)$ is a linear combination of the coefficients of $\G_T$, as desired.
\end{proof}

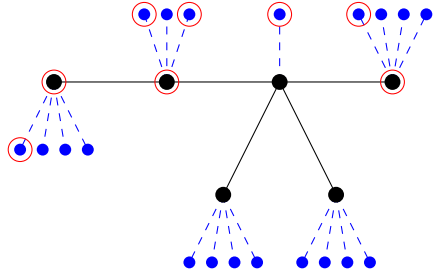
\begin{figure}
    \centering
    \begin{tikzpicture}[scale=1.5]
        \coordinate (A) at (0,0);
        \coordinate (B) at (1,0);
        \coordinate (C) at (-1,0);
        \coordinate (D) at (.5,-1);
        \coordinate (E) at (1.5,-1);
        \coordinate (F) at (2,0);
        \coordinate (A1) at (-.2,.6);
        \coordinate (A2) at (0,.6);
        \coordinate (A3) at (.2,.6);
        \coordinate (B1) at (1,.6);
        \coordinate (C1) at (-1.3,-.6);
        \coordinate (C2) at (-1.1,-.6);
        \coordinate (C3) at (-.9,-.6);
        \coordinate (C4) at (-.7,-.6);
        \coordinate (D1) at (.2,-1.6);
        \coordinate (D2) at (.4,-1.6);
        \coordinate (D3) at (.6,-1.6);
        \coordinate (D4) at (.8,-1.6);
        \coordinate (E1) at (1.2,-1.6);
        \coordinate (E2) at (1.4,-1.6);
        \coordinate (E3) at (1.6,-1.6);
        \coordinate (E4) at (1.8,-1.6);
        \coordinate (F1) at (1.7,.6);
        \coordinate (F2) at (1.9,.6);
        \coordinate (F3) at (2.1,.6);
        \coordinate (F4) at (2.3,.6);
        \draw (C)--(A)--(B)--(F); \draw (D)--(B)--(E);
        \draw[dashed, blue] (A)--(A1);
        \draw[dashed, blue] (A)--(A2);
        \draw[dashed, blue] (A)--(A3);
        \draw[dashed, blue] (B)--(B1);
        \draw[dashed, blue] (C)--(C1);
        \draw[dashed, blue] (C)--(C2);
        \draw[dashed, blue] (C)--(C3);
        \draw[dashed, blue] (C)--(C4);
        \draw[dashed, blue] (D)--(D1);
        \draw[dashed, blue] (D)--(D2);
        \draw[dashed, blue] (D)--(D3);
        \draw[dashed, blue] (D)--(D4);
        \draw[dashed, blue] (E)--(E1);
        \draw[dashed, blue] (E)--(E2);
        \draw[dashed, blue] (E)--(E3);
        \draw[dashed, blue] (E)--(E4);
        \draw[dashed, blue] (F)--(F1);
        \draw[dashed, blue] (F)--(F2);
        \draw[dashed, blue] (F)--(F3);
        \draw[dashed, blue] (F)--(F4);
        \fill [black] (A) circle(2.5pt);
        \fill [black] (B) circle(2.5pt);
        \fill [black] (C) circle(2.5pt);
        \fill [black] (D) circle(2.5pt);
        \fill [black] (E) circle(2.5pt);
        \fill [black] (F) circle(2.5pt);
        \fill [blue] (A1) circle(1.5pt);
        \fill [blue] (A2) circle(1.5pt);
        \fill [blue] (A3) circle(1.5pt);
        \fill [blue] (B1) circle(1.5pt);
        \fill [blue] (C1) circle(1.5pt);
        \fill [blue] (C2) circle(1.5pt);
        \fill [blue] (C3) circle(1.5pt);
        \fill [blue] (C4) circle(1.5pt);
        \fill [blue] (D1) circle(1.5pt);
        \fill [blue] (D2) circle(1.5pt);
        \fill [blue] (D3) circle(1.5pt);
        \fill [blue] (D4) circle(1.5pt);
        \fill [blue] (E1) circle(1.5pt);
        \fill [blue] (E2) circle(1.5pt);
        \fill [blue] (E3) circle(1.5pt);
        \fill [blue] (E4) circle(1.5pt);
        \fill [blue] (F1) circle(1.5pt);
        \fill [blue] (F2) circle(1.5pt);
        \fill [blue] (F3) circle(1.5pt);
        \fill [blue] (F4) circle(1.5pt);
        \draw[red] (A) circle(4pt);
        \draw[red] (C) circle(4pt);
        \draw[red] (F) circle(4pt);
        \draw[red] (A1) circle(3pt);
        \draw[red] (A3) circle(3pt);
        \draw[red] (B1) circle(3pt);
        \draw[red] (C1) circle(3pt);
        \draw[red] (F1) circle(3pt);
    \end{tikzpicture}
    \caption{An example illustrating the proof of Proposition~\ref{prop:gdp-augmentation}. Solid black edges and black vertices are in $T$; dashed blue edges and smaller blue vertices are added to get $U$; circled vertices are in $A$. The parameters have values $(n,k) = (6,5)$, $(a,b,c) = (8,11,4)$, $(\ell, i) = (4,3)$, $(d,e) = (2,1)$, and $\ell^* = 20$.}
    \label{fig:gdp-augmentation-proof}
\end{figure}

Thus if $T_1$ and $T_2$ are the trees in Figure~\ref{fig:same-gdp}, then for all $k \ge 4$, the trees $T_1^{[k]}$ and $T_2^{[k]}$ also have the same generalized degree polynomial. (In fact, we can take repeated augmentations: for example, $(T_1^{[k]})^{[\ell]}$ and $(T_2^{[k]})^{[\ell]}$ have the same generalized degree polynomial for $\ell \ge k \ge 4$.)

\section{Degree embeddings} \label{sec:degree-embeddings}

Our main result concerns several classes of information about a tree $T$ that are determined by $\G_T$. This information can be organized using the idea of \emph{degree embeddings}, which we now define.

\begin{defn}
    Let $D$ be a tree whose vertices $u$ are labeled with positive integers $a_u$. A \emph{degree embedding} of $D$ into $T$ is an isomorphism $\varphi \colon D \to T'$, where $T'$ is a subtree of $T$, such that $a_u = \deg_T (\varphi(u))$ for all $u \in V(D)$. We define the \emph{degree embedding number} of $D$ in $T$, denoted $N_D(T)$, to be the number of subtrees $T' \subseteq T$ for which a degree embedding exists with image $T'$.
\end{defn}

We adopt a simplified notation when $D$ is a path, say with vertex labels $a_1, \dots, a_m$ in that order. In this case, $N_D(T)$ is the number of paths $(v_1, \dots, v_m)$ in $T$ such that $a_i = \deg(v_i)$ for $i = 1, \dots, m$. (Note that if the sequence $a_1, \dots, a_m$ is palindromic, then these paths should only be counted up to reversal.) Then we will write $N_{a_1, \dots, a_m}(T)$ to mean $N_D(T)$. In particular, $N_a(T)$ is the number of vertices of $T$ with degree $a$.

\begin{example} \label{ex:degree-embeddings}
    For the tree $T$ depicted in Figure~\ref{fig:degree-embedding-example}, we have \[
        N_1(T) = 9, \qquad
        N_{2,2}(T) = 1, \qquad
        N_{1,3,1}(T) = 1, \quad \text{and} \quad
        N_D(T) = 4,
    \]
    where $D$ is the star on four vertices with center labeled $4$ and outer vertices labeled $1, 1, 4$.
    
    \begin{figure}
        \centering
        \begin{tikzpicture}[scale=1.25]
            \coordinate[label=below:1] (A) at (0,0);
            \coordinate[label=below:2] (B) at (1,0);
            \coordinate[label=below:2] (C) at (2,0);
            \coordinate[label=below:4] (D) at (3,0);
            \coordinate[label=below:4] (E) at (4,0);
            \coordinate[label=below:3] (F) at (5,0);
            \coordinate[label=below:1] (G) at (6,0);
            \coordinate[label=above:1] (H) at (3.7,1);
            \coordinate[label=above:1] (I) at (4.3,1);
            \coordinate[label=above:1] (J) at (5,1);
            \coordinate[label=above:1] (K) at (3.3,1);
            \coordinate[label=left:4] (L) at (2.7,1);
            \coordinate[label=above:1] (M) at (2.2,2);
            \coordinate[label=above:1] (N) at (2.7,2);
            \coordinate[label=above:1] (O) at (3.2,2);
            \draw (A)--(B)--(C)--(D)--(E)--(F)--(G);
            \draw (I)--(E)--(H); \draw (F)--(J);
            \draw (K)--(D)--(L)--(M); \draw (N)--(L)--(O);
            \fill [black] (A) circle(2pt);
            \fill [black] (B) circle(2pt);
            \fill [black] (C) circle(2pt);
            \fill [black] (D) circle(2pt);
            \fill [black] (E) circle(2pt);
            \fill [black] (F) circle(2pt);
            \fill [black] (G) circle(2pt);
            \fill [black] (H) circle(2pt);
            \fill [black] (I) circle(2pt);
            \fill [black] (J) circle(2pt);
            \fill [black] (K) circle(2pt);
            \fill [black] (L) circle(2pt);
            \fill [black] (M) circle(2pt);
            \fill [black] (N) circle(2pt);
            \fill [black] (O) circle(2pt);
        \end{tikzpicture}
        \caption{The tree $T$ of Example~\ref{ex:degree-embeddings}. Vertices are labeled with their degrees.}
        \label{fig:degree-embedding-example}
    \end{figure}
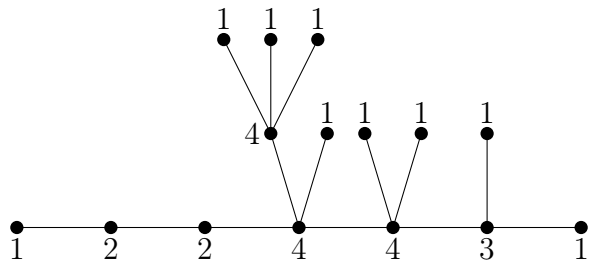
    
\end{example}

Clearly, a tree $T$ can be recovered from all its degree embedding numbers because $T$ (with vertices labeled by degree) is the unique largest tree $D$ with $N_D(T) \neq 0$. However, knowing any ``bounded'' collection of degree embedding numbers is not sufficient in general to recover $T$, as we now show.

\begin{proposition} \label{prop:degree-embedding-same}
    For every $m \ge 1$, there exist non-isomorphic trees $T_1$ and $T_2$ such that $N_D(T_1) = N_D(T_2)$ for all $D$ with $|D| \le m$.
\end{proposition}

\begin{proof}
    Let $T_1$ be the tree created by attaching leaves to a path with $2m+4$ vertices so that the path vertices have degrees \[4, \underbrace{2, \dots, 2}_m, 3, \underbrace{2, \dots, 2}_{m+1}, 5\] in that order. Let $T_2$ be the tree constructed analogously but using degrees \[4, \underbrace{2, \dots, 2}_{m+1}, 3, \underbrace{2, \dots, 2}_m, 5.\] These trees are not isomorphic; however, we can give a bijection between subtrees $S_1 \subseteq T_1$ and $S_2 \subseteq T_2$ of order $|S_1| = |S_2| \le m$ that preserves vertex degrees (with respect to $T_1$ and $T_2$). Suppose $S_1$ is a subtree of $T_1$ with $|S_1| \le m$. If $S_1$ contains the degree-$4$ vertex of $T_1$, then it contains at most $m-1$ degree-$2$ vertices on the left side of $T_1$. It follows that there exists a corresponding subtree $S_2$ on the left side of $T_2$ with the same vertex degrees. A similar argument holds for the case in which $S_1$ contains the degree-$5$ vertex of $T_1$.
    
    Now the remaining subtrees of $T_1$ and $T_2$ are those not containing the degree-$4$ or degree-$5$ vertices. But when the degree-$4$ and degree-$5$ vertices of $T_1$ and $T_2$ are deleted, the resulting graphs are isomorphic (under an isomorphism that preserves the original vertex degrees), so a bijection between subtrees exists as desired. Thus, $N_D(T_1) = N_D(T_2)$ whenever $|D| \le m$.
\end{proof}

As a preview of the next section, we now present a list of degree embedding numbers that we will be able to recover from $\G_T$. (Note, however, that many degree embedding numbers cannot be determined by $\G_T$. For example, in Figure~\ref{fig:same-gdp}, let $T_1$ be the tree on the left and let $T_2$ be the tree on the right. Then $\G_{T_1} = \G_{T_2}$, but one can check that $N_{1,2,3}(T_1) = 1 \neq 0 = N_{1,2,3}(T_2)$ and $N_{1,3,2}(T_1) = 2\neq 3 = N_{1,3,2}(T_2)$.)

\begin{theorem} \label{thm:degree-embedding-list}
    The generalized degree polynomial $\G_T$ determines the degree embedding numbers $N_D(T)$ when $D$ is any of the following:
    \begin{enumerate}[label=(\roman*)]
        \item \label{thm-part:Nab} A path of order $2$.
        \item \label{thm-part:Naaa} A path of order $3$ with all vertex labels equal.
        \item \label{thm-part:Naab} A path of order $3$ with vertex labels $a, a, b$ in that order, where $b \ge 2a$.
        \item \label{thm-part:N2...2} A path in which all vertices are labeled $2$.
        \item \label{thm-part:leaf-adjacency} A star in which the outer vertices are all labeled $1$.
    \end{enumerate}
\end{theorem}

For completeness, we note that $\G_T$ determines $N_D(T)$ if $D$ is a single vertex, since we already know that $\G_T$ determines the degree sequence of $T$. Additionally, $\G_T$ determines $N_D(T)$ if $D$ is a path with vertex labels $1, 2, 2, \dots, 2$ in that order; this is a consequence of a result of Crew \cite[Theorem 3]{crewnote} about the \emph{twigs} of a tree, as mentioned in the introduction.

\section{Degree embedding numbers from the generalized degree polynomial} \label{sec:degree-embedding-numbers}

To prove Theorem~\ref{thm:degree-embedding-list}, we first construct a sequence of two-variable rational functions $(\Gk)_{k \ge 1}$ which, taken together, are equivalent to $\G_T$. Then, by working with these rational functions, we recover all the desired degree embedding numbers $N_D(T)$.

\subsection{The rational functions $\Gk$} \label{subsec:Gk}

To define $\Gk$, we will need the following notion. For a subset of edges $S \subseteq E(T)$, let $\verts{S}$ be the subset of $V(T)$ consisting of all vertices incident to at least one edge in $S$. (That is, considering each edge as a set of size 2, we have $\verts{S} = \bigcup_{e \in S} e$.)

\begin{defn}
    Let $\Ga0(x,y) = 1$, and for positive integers $k$, define \[
        \Gk(x,y) = \sum_{\substack{S \subseteq E(T) \\ |S| = k}} \prod_{v \in \verts{S}} \frac{xy^{\deg(v)}}{1+xy^{\deg(v)}}.
    \]
\end{defn}

\begin{example}
    Let $T$ be the path on $6$ vertices. For brevity, let $q = \frac{xy}{1+xy}$ and $r = \frac{xy^2}{1+xy^2}$. Then $\Ga1(x,y) = 2qr + 3r^2$ because two edges of $T$ connect a leaf and a vertex of degree $2$, and the other three edges connect two vertices of degree $2$. We also have \[
        \Ga2(x,y) = q^2r^2 + 4qr^3 + r^4 + 2qr^2 + 2r^3.
    \] For example, the $2r^3$ term arises because there are two sets $S \subseteq E(T)$ with $|S| = 2$ such that $\verts{S}$ consists of three vertices of degree $2$.
\end{example}

We now show how $\Gk$ can be computed from $\G_T$. The key is to take partial derivatives of $\G_T(x,y,z)$ with respect to $z$.

\begin{proposition}
    For each $k$, $\Gk(x,y)$ is determined by $\G_T(x,y,z)$. Specifically, we have
    \begin{align} \label{eq:Gk}
        \Gk(x,y) = \frac{y^{2k}}{k!} \cdot \frac{\d^k \G_T}{\d z^k}(x,y,y^2) \cdot \G_T(x, y, y^2)^{-1}.
    \end{align}
\end{proposition}

\begin{proof}
    Since $\G_T(x,y,z) = \sum_{A \subseteq V(T)} x^{|A|} y^{d(A)} z^{e(A)}$, we have
    \[
        \frac{z^k}{k!} \cdot \frac{\d^k \G_T}{\d z^k}(x,y,z)
        = \sum_{A \subseteq V(T)} \binom{e(A)}{k} x^{|A|} y^{d(A)} z^{e(A)}
        = \sum_{A \subseteq V(T)} \sum_{\substack{S \subseteq E(A) \\ |S| = k}} x^{|A|} y^{d(A)} z^{e(A)}.
    \] Next, we set $z = y^2$, which yields \[\begin{aligned}
        \frac{y^{2k}}{k!} \cdot \frac{\d^k \G_T}{\d z^k}(x,y,y^2)
        &= \sum_{A \subseteq V(T)} \sum_{\substack{S \subseteq E(A) \\ |S| = k}} x^{|A|} y^{d(A)+2e(A)}
        = \sum_{A \subseteq V(T)} \sum_{\substack{S \subseteq E(A) \\ |S| = k}} \prod_{v \in A} xy^{\deg(v)}
    \end{aligned}\] by the identity \eqref{eq:deg-identity}. We now reverse the order of summation. Observe that for $A \subseteq V(T)$ and $S \subseteq E(T)$, we have $S \subseteq E(A)$ if and only if $\verts{S} \subseteq A$. Therefore we can write
    \[\begin{aligned}
        \frac{y^{2k}}{k!} \cdot \frac{\d^k \G_T}{\d z^k}(x,y,y^2)
        &=
        \sum_{\substack{S \subseteq E(T) \\ |S| = k}} \sum_{\verts{S} \subseteq A \subseteq V(T)} \prod_{v \in A} xy^{\deg(v)}
        \\ &=
        \sum_{\substack{S \subseteq E(T) \\ |S| = k}} \bigg(\, \prod_{v \in \verts{S}} xy^{\deg(v)} \prod_{v \in V(T) \setminus \verts{S}} \left(1 + xy^{\deg(v)}\right) \bigg)
        \\ &=
        \prod_{v \in V(T)} (1 + xy^{\deg(v)}) \;\cdot\; \sum_{\substack{S \subseteq E(T) \\ |S| = k}} \prod_{v \in \overline{S}} \frac{xy^{\deg(v)}}{1+xy^{\deg(v)}},
    \end{aligned}\]
    which shows that
    \[
        \frac{y^{2k}}{k!} \cdot \frac{\d^k \G_T}{\d z^k}(x,y,y^2)
        =
        \prod_{v \in V(T)} (1 + xy^{\deg(v)}) \;\cdot\; \Gk(x,y).
    \]
    By taking $k = 0$, we see that $\G_T(x,y,y^2) = \prod_{v \in V(T)} (1+xy^{\deg(v)})$, so \eqref{eq:Gk} follows.
\end{proof}

Note that, conversely, $\G_T$ can be recovered from all of the $\Gk$. Indeed, the Taylor series expansion of $\G_T(x,y,z)$ about $z = y^2$ is
\[\begin{aligned}
    \G_T(x,y,z)
    &= \sum_{k \ge 0} \frac{1}{k!} \cdot \frac{\d^k \G_T}{\d z^k}(x,y,y^2) \cdot (z-y^2)^k
    \\ &= \G_T(x,y,y^2) \sum_{k \ge 0} \frac{1}{y^{2k}} \cdot \Gk(x,y) \cdot (z-y^2)^k,
\end{aligned}\] so the $\Gk$ determine the product $\G_T(x,y,z) \cdot \G_T(x,y,y^2)^{-1}$. We will see as a consequence of the results of Section~\ref{subsec:double-degrees} that the $\Gk$ determine the degree sequence of $T$, so they determine $\G_T(x,y,y^2)$ and thus $\G_T(x,y,z)$.

Before proceeding, we comment on the relationship between $\G_T$ and the $\Gk$. Whereas $\X_T$ determines $\G_T$ linearly, it is not true that $\G_T$ determines $\Gk$ linearly. In other words, treating $\Gk$ as a formal power series in $x$ and $y$, its coefficients are generally not linear functions of the coefficients $g_T(a,b,c)$ of $\G_T$ (as one can check computationally); rather, the formula \eqref{eq:Gk} implies that each coefficient of $\Gk$ can be written as a polynomial in the $g_T(a,b,c)$. We also remark that some of the methods we will use to extract degree embedding numbers from the $\Gk$ are highly nonlinear (particularly those in Sections~\ref{subsec:matchings} and \ref{subsec:leaf-adjacency}).

\subsection{Double-degree numbers} \label{subsec:double-degrees}

We now prove the first part of Theorem~\ref{thm:degree-embedding-list} by showing that $\Ga1$, and thus $\G_T$, determines all the degree embedding numbers $N_{a,b}(T)$. (We call these \emph{double-degree numbers}.) Observe that by a double-counting argument, \[
    N_a(T) = \frac{1}{a}\left(\sum_{b \neq a} N_{a,b}(T) + 2N_{a,a}(T)\right).
\] Thus the double-degree numbers determine the degree sequence of $T$, as promised above.

\begin{theorem} \label{thm:double-degrees}
    The rational function $\Ga1$ determines all double-degree numbers $N_{a,b}(T)$.
\end{theorem}

\begin{proof}
    By symmetry, we may consider only the double-degree numbers $N_{a,b}(T)$ where $a \le b$. We proceed by induction on the pairs $(a,b)$ with $a \le b$ in lexicographical order. The base case $(a,b) = (1,1)$ is trivial because $N_{1,1}(T) = 1$ if $|T| = 2$ and $N_{1,1}(T) = 0$ otherwise. For the induction step, fix positive integers $a_0 \le b_0$ with $b_0 \ge 2$; assume that $\Ga1$ determines all $N_{a,b}(T)$ (with $a \le b$) where $(a,b)$ precedes $(a_0,b_0)$ lexicographically. We expand $\Ga1$ as a formal power series as follows:
    \begin{align}
        \Ga1(x,y)
        &= \sum_{\{v,w\} \in E(T)} \frac{xy^{\deg(v)}}{1+xy^{\deg(v)}} \frac{xy^{\deg(w)}}{1+xy^{\deg(w)}} \nonumber
        \\ &= \sum_{a \le b} \frac{xy^a}{1+xy^a} \frac{xy^b}{1+xy^b} N_{a,b}(T) \nonumber
        \\ &= \sum_{a \le b} \sum_{k, \ell \ge 1} (-1)^{k+\ell} (xy^a)^k (xy^b)^\ell N_{a,b}(T). \label{eq:Ga1}
    \end{align}
    Consider the coefficient $C$ of the monomial $x^{b_0} y^{a_0b_0-a_0+b_0}$ in this power series. From \eqref{eq:Ga1}, we see that $C$ is a linear combination of terms $N_{a,b}(T),$ where $a \le b$ satisfy \[
        k+\ell = b_0
        \qquad
        \text{and}
        \qquad
        ka+\ell b = a_0b_0-a_0+b_0
    \] for some $k, \ell \ge 1$. (The sign $(-1)^{k+\ell} = (-1)^{b_0}$ is constant for all such terms, so there is no cancellation.) These equations are satisfied when $(a, b, k, \ell) = (a_0, b_0, b_0-1, 1)$, so $N_{a_0,b_0}(T)$ appears as a term in the expression for $C$. We claim that for every other term $N_{a,b}(T)$ that appears, $(a,b)$ precedes $(a_0,b_0)$ lexicographically. Indeed, we have \[
        ab_0 = (k+\ell)a \le ka + \ell b = a_0b_0 - a_0 + b_0 < (a_0+1)b_0,
    \] so $a < a_0 + 1$ and thus $a \le a_0$. Furthermore, if $a = a_0$, then we have \[
        \ell b = (a_0b_0 - a_0 + b_0) - ka = (a_0b_0 - a_0 + b_0) - (b_0 - \ell)a_0 = (\ell-1)a_0 + b_0 \le \ell b_0,
    \] so $b \le b_0$. This proves the claim. It then follows from the induction hypothesis that $\Ga1$ determines $N_{a_0,b_0}(T)$.
\end{proof}

This establishes Theorem~\ref{thm:degree-embedding-list}\ref{thm-part:Nab}.

\begin{example}
    As an illustration of the above proof, we derive a formula for $N_{2,3}(T)$ in terms of coefficients of $\Ga1(x,y)$. For $(a_0, b_0) = (2,3)$, we have $x^{b_0} y^{a_0b_0 - a_0 + b_0} = x^3 y^7$, and so using \eqref{eq:Ga1} we compute \[
        [x^3 y^7] \Ga1 = N_{1,3}(T) + N_{1,5}(T) + N_{2,3}(T).
    \] It then remains to compute $N_{1,3}(T)$ and $N_{1,5}(T)$. For $(a_0,b_0) = (1,5)$, we get \[
        [x^5 y^9] \Ga1 = N_{1,2}(T) + N_{1,3}(T) + N_{1,5}(T),
    \] so (by coincidence) the $N_{1,3}(T)$ terms will cancel and we need only compute $N_{1,2}(T)$. For $(a_0,b_0) = (1,2)$, we have just \[
        [x^2 y^3] \Ga1 = N_{1,2}(T).
    \] Therefore, we get the formula \[
        N_{2,3}(T) = [x^3 y^7] \Ga1 - [x^5 y^9] \Ga1 + [x^2 y^3] \Ga1.
    \]
\end{example}

Using this procedure, any double-degree number $N_{a,b}(T)$ can be expressed as a linear combination of coefficients of $\Ga1$ and hence as a polynomial in the coefficients of $\G_T$.

\subsection{Matchings and paths} \label{subsec:matchings}

We now prove parts \ref{thm-part:Naaa} and \ref{thm-part:N2...2} of Theorem~\ref{thm:degree-embedding-list}. To do so, recall that for a graph $G$, the \emph{matching-generating polynomial} of $G$ is the univariate polynomial \[
    M_G(t) = \sum_{S} t^{|S|},
\] where $S$ ranges over all matchings (collections of disjoint edges) of $G$. For $d \ge 2$, denote by $T_{(d)}$ the induced subgraph of $T$ consisting of all degree-$d$ vertices of $T$.

\begin{theorem} \label{thm:matching-polynomial}
    The polynomial $\G_T$ determines $M_{T_{(d)}}(t)$ for all $d$.
\end{theorem}

\begin{proof}
    Let $k$ be an arbitrary positive integer. Consider the product \[
        (1+xy^d)^{2k} \cdot \Gk(x,y) = (1+xy^d)^{2k} \sum_{\substack{S \subseteq E(T) \\ |S| = k}} \prod_{v \in \verts{S}} \frac{xy^{\deg(v)}}{1+xy^{\deg(v)}}.
    \] For any set of edges $S \subseteq E$ with $|S| = k$, we have $\left|\verts{S}\right| \le 2k$, with equality if and only $S$ is a matching. Therefore, if we set $x = -y^{-d}$ (so $1+xy^d = 0$), the only terms that survive are those having $2k$ factors of $1+xy^d$ in the denominator. These terms occur exactly when $S$ is a matching of size $k$ of $T_{(d)}$. Each such term evaluates to $(-1)^{2k}=1$, so we recover the number of such matchings for each $k$. Since $k$ was arbitrary and $\G_T$ determines all $\Gk$, it follows that $\G_T$ determines $M_{T_{(d)}}(t)$.
\end{proof}

Now, observe that the number of matchings of size $2$ in $T_{(d)}$ is given by $\binom{N_{d,d}(T)}{2} - N_{d,d,d}(T)$. Because $\G_T$ determines $N_{d,d}(T)$ from Section~\ref{subsec:double-degrees}, it follows that $\G_T$ determines $N_{d,d,d}(T)$ as well. This proves Theorem~\ref{thm:degree-embedding-list}\ref{thm-part:Naaa}.

To establish Theorem~\ref{thm:degree-embedding-list}\ref{thm-part:N2...2}, we will show that $\G_T$ determines $T_{(2)}$ (up to isomorphism). Note that $T_{(2)}$ is a disjoint union of paths since all of its vertices have degree at most $2$. Furthermore, $\G_T$ determines the order of $T_{(2)}$ (because this is simply $N_2(T)$) as well as the matching-generating polynomial $M_{T_{(2)}}(t)$. This information suffices to determine $T_{(2)}$ because of the following elementary result.

\begin{proposition}
    If $G$ is a disjoint union of paths, then $G$ is determined by $|G|$ and $M_G(t)$.
\end{proposition}

\begin{proof}
    For $j \ge 0$, let $f_j(t)$ be the matching-generating polynomial of a path on $j$ vertices. Because the matching-generating polynomial multiplies over disjoint unions, it follows that \begin{align} \label{eq:MGt}
        M_G(t) = f_2(t)^{c_2} \dotsm f_n(t)^{c_n}
    \end{align} where $n = |G|$ and $c_j \ge 0$ is the number of components of $G$ of order $j$. (Note that $f_1(t) = 1$, so there is no $f_1(t)$ term.) We will show that all products of this form are distinct. The combinatorial definition of $f_j(t)$ yields the recurrence relation $f_j(t) = f_{j-1}(t) + t f_{j-2}(t)$ for $j \ge 2$, with $f_0(t) = f_1(t) = 1$. One can then show by induction that \[
        U_j(x) = (2x)^j f_j\left(-\tfrac{1}{4x^2}\right),
    \] where $U_j(x)$ denotes the $j$th Chebyshev polynomial of the second kind. The roots of $U_j(x)$ are the numbers $x = \cos(\frac{k\pi}{j+1})$ for $k = 1, 2, \dots, j$, so it follows that $f_j(t)$ has distinct roots \[
        t = -\frac{1}{4 \cos^2(\frac{k\pi}{j+1})}
    \] for $1 \le k < \frac{j+1}{2}$, a total of $\lfloor \frac{j}{2} \rfloor$ roots. Since $f_j$ has degree $\lfloor \frac{j}{2} \rfloor$, these are all of its roots. In particular, the largest root of $f_j(t)$ is $t_j = -1/(4 \cos^2 \frac{\pi}{j+1})$, which increases with $j$, so $t_j$ is not a root of any of $f_2(t), \dots, f_{j-1}(t)$. It then follows that all products of the form \eqref{eq:MGt} are distinct: if \[
        f_2(t)^{c_2} \dotsm f_n(t)^{c_n} = f_2(t)^{d_2} \dotsm f_n(t)^{d_n}
    \] for some $c_j, d_j \ge 0$, then considering the order of the root $t_n$ on both sides forces $c_n = d_n$, and one can continue by induction to show that $c_j = d_j$ for all $j$.

    Therefore, $M_G(t)$ determines the number of components of $G$ of order $j$ for all $j \ge 2$. Then, $|G|$ determines the number of components of order $1$, so $G$ is fully determined.
\end{proof}

It now follows that $\G_T$ determines $N_{\underbrace{\scriptstyle 2, \dots, 2}_k}(T)$ for all $k$, because this is the number of paths of order $k$ in $T_{(2)}$. Thus, we have established Theorem~\ref{thm:degree-embedding-list}\ref{thm-part:N2...2}.

\subsection{Triple-degree numbers} \label{subsec:triple-degrees}

We now consider degree embedding numbers of the form $N_{a,b,c}(T)$, which we call \emph{triple-degree numbers}. We will recover certain linear combinations of triple-degree numbers from $\G_T$, including $N_{a,a,b}(T)$ for $b \ge 2a$, which proves Theorem~\ref{thm:degree-embedding-list}\ref{thm-part:Naab}.

\begin{theorem}
    Let $a \ge 2$ be an integer. Then $\G_T$ determines all of the following:
    \begin{enumerate}[label=(\roman*)]
        \item $N_{a,a,1}(T) + N_{a,a,2}(T) + \cdots + N_{a,a,a-1}(T)$,
        \item $N_{a,a,a-k}(T) - N_{a,a,a+k}(T)$ for all $1 \le k \le a-1$, and
        \item $N_{a,a,b}(T)$ for all $b \ge 2a$.
    \end{enumerate}
\end{theorem}

\begin{proof}
    Throughout this proof, we write $N_D$ instead of $N_D(T)$ for brevity. We know that $\G_T$ determines the rational function \[
        \Ga2(x, y) = \sum_{\substack{S \subseteq E(T) \\ |S| = 2}} \prod_{v \in \verts{S}} \frac{xy^{\deg(v)}}{1+xy^{\deg(v)}}.
    \]
    For some of the terms in this sum, we have $S \subseteq E(T_{(a)})$, that is, all the vertices in $\verts{S}$ have degree $a$. Each such term equals either $(\frac{xy^a}{1+xy^a})^4$ or $(\frac{xy^a}{1+xy^a})^3$ depending on whether $S$ is a matching or not, respectively. By Theorem~\ref{thm:matching-polynomial}, $\G_T$ determines the number of terms of each type. Therefore, it follows that $\G_T$ determines the sum \[
        \til{\Gamma}(x,y) = \sum_{\substack{S \subseteq E(T) \\ S \nsubseteq E(T_{(a)}) \\ |S| = 2}} \prod_{v \in \verts{S}} \frac{xy^{\deg(v)}}{1+xy^{\deg(v)}}
    \] obtained by deleting all terms in $\Ga2(x,y)$ where $S \subseteq E(T_{(a)})$.
    
    In each term of $\til{\Gamma}(x,y)$, the number of factors of $1+xy^a$ in the denominator is at most $3$. Furthermore, equality holds exactly when the two edges of $S$ are disjoint, the endpoints of one edge both have degree $a$, and the endpoints of the other edge have degrees $a$ and $b$ for some $b \neq a$. Then the corresponding term of $\til{\Gamma}(x,y)$ is $(\frac{xy^a}{1+xy^a})^3 \frac{xy^b}{1+xy^b},$ and the number of such terms for a given $b$ is $N_{a,a} N_{a,b} - N_{a,a,b}$. Therefore if we take $x = -y^{-a}$ in the product $(1+xy^a)^3 \cdot \til{\Gamma}(x,y)$, then the result is \[\begin{aligned}
        & \sum_{b \neq a} (xy^a)^3 \frac{xy^b}{1+xy^b} \cdot \left( N_{a,a} N_{a,b} - N_{a,a,b} \right) \bigg|_{x = -y^{-a}}
        \\ &= \sum_{b \neq a} \frac{y^{b-a}}{1-y^{b-a}} \cdot \left( N_{a,a} N_{a,b} - N_{a,a,b} \right).
    \end{aligned}\] Because $\G_T$ determines all $N_{a,a}$ and $N_{a,b}$ by Theorem~\ref{thm:double-degrees}, it follows that $\G_T$ also determines \[
        \sum_{b \neq a} \frac{y^{b-a}}{1-y^{b-a}} N_{a,a,b}.
    \] Now we write \[\begin{aligned}
        \sum_{b \neq a} \frac{y^{b-a}}{1-y^{b-a}} N_{a,a,b}
        &= -\sum_{b=1}^{a-1} \frac{1}{1-y^{a-b}} N_{a,a,b} + \sum_{b=a+1}^\infty \frac{y^{b-a}}{1-y^{b-a}} N_{a,a,b}
        \\&= -\sum_{k=1}^{a-1} \frac{1}{1-y^k} N_{a,a,a-k} + \sum_{k=1}^\infty \frac{y^k}{1-y^k} N_{a,a,a+k}
        \\&= -\sum_{k=1}^{a-1} \left(1 + \frac{y^k}{1-y^k}\right) N_{a,a,a-k} + \sum_{k=1}^\infty \frac{y^k}{1-y^k} N_{a,a,a+k}
        \\&= -\sum_{k=1}^{a-1} N_{a,a,a-k}  - \sum_{k=1}^{a-1} \frac{y^k}{1-y^k} (N_{a,a,a-k} - N_{a,a,a+k}) + \sum_{k=a}^\infty \frac{y^k}{1-y^k} N_{a,a,a+k}.
    \end{aligned}\] The rational functions $1$ and $\frac{y^k}{1-y^k}$ for $k \ge 1$ are linearly independent over $\Q$ because they have roots of different orders at $y=0$. Thus, $\G_T$ determines the coefficients of these rational functions in the above expression, which are precisely the desired linear combinations.
\end{proof}

\subsection{Leaf adjacency} \label{subsec:leaf-adjacency}

Finally, we prove Theorem~\ref{thm:degree-embedding-list}\ref{thm-part:leaf-adjacency}. Assume that $|T| \ge 3$, so that no two leaves of $T$ are adjacent. For a vertex $v \in V(T)$, let $\ell(v)$ be the number of leaves of $T$ adjacent to $v$, and let $L_{d,\ell}(T)$ be the number of vertices $v$ with $\deg(v) = d$ and $\ell(v) = \ell$. The numbers $L_{d,\ell}(T)$ form the \emph{leaf adjacency sequence} of $T$.

\begin{theorem}
    Fix integers $d \ge 2$ and $\ell \ge 0$. Then $\G_T$ determines $L_{d,\ell}(T)$. 
\end{theorem}
    
\begin{proof}
    Fix a positive integer $k$. Instead of $\Gk(x,y)$, it will be more convenient to make the substitution \[
        \Gk(x^{-1},y^{-1}) = \sum_{\substack{S \subseteq E(T) \\ |S| = k}} \prod_{v \in \verts{S}} \frac{1}{1+xy^{\deg(v)}}.
    \] For each term in the above sum, the number of factors of $1+xy$ in the denominator is at most $k$, with equality if and only if $S \subseteq L(T)$, where $L(T)$ is the set of all edges of $T$ adjacent to a leaf. Therefore, we have \[
        (1+xy)^k \cdot \Gk(x^{-1},y^{-1}) \bigg|_{x = -y^{-1}} = 
        \sum_{\substack{S \subseteq L(T) \\ |S| = k}} \prod_{v \in \nlverts{S}} \frac{1}{1-y^{\deg(v)-1}},
    \] where $\nlverts{S}$ denotes the set of all non-leaf vertices in $\verts{S}$. (Note that $\nlverts{S}$ has size at most $k$.) Now fix an integer $a \ge 2$, and consider the product \[
        (1-y^{a-1})^k \sum_{\substack{S \subseteq L(T) \\ |S| = k}} \prod_{v \in \nlverts{S}} \frac{1}{1-y^{\deg(v)-1}}.
    \] Let $\omega = e^{2\pi i/(a-1)}$, which is a root of $1 - y^{d-1}$ if and only if $(a-1) \mid (d-1)$, in which case it is a simple root. If we set $y = \omega$ in the above product, the only terms that will survive are those for which $\nlverts{S}$ has size $k$ (thus, all the edges of $S$ are disjoint), and every vertex $v \in \nlverts{S}$ satisfies $(a-1) \mid (\deg(v)-1)$. Let $V^{(a)}$ be the set of all such vertices: \[
        V^{(a)} = \{ v \in V(T) : \deg(v) > 1 \; \text{and} \; (a-1) \mid (\deg(v)-1) \}.
    \] Then each surviving term corresponds to a choice of $k$ vertices from $V^{(a)}$ and a leaf edge incident to each chosen vertex. Therefore, setting $y = \omega$ yields \[\begin{aligned}
        (1-y^{a-1})^k \sum_{\substack{A \subseteq V^{(a)} \\ |A|=k}} \prod_{v \in A} \left(\ell(v) \frac{1}{1-y^{\deg(v)-1}} \right)
        \bigg|_{y = \omega}
        &=
        \sum_{\substack{A \subseteq V^{(a)} \\ |A|=k}} \prod_{v \in A} \left( \ell(v) \lim_{y \to \omega} \frac{1-y^{a-1}}{1-y^{\deg(v)-1}} \right)
        \\ &=
        \sum_{\substack{A \subseteq V^{(a)} \\ |A|=k}} \prod_{v \in A} \left( \ell(v) \frac{a-1}{\deg(v)-1} \right)
        \\ &=
        (a-1)^k \sum_{\substack{A \subseteq V^{(a)} \\ |A|=k}} \prod_{v \in A} \frac{\ell(v)}{\deg(v)-1}.
    \end{aligned}\] Ignoring the factor of $(a-1)^k$, this is the $k$th elementary symmetric sum of the multiset \[
        M^{(a)} = \left\{ \frac{\ell(v)}{\deg(v)-1} : v \in V^{(a)} \right\}.
    \] Since $k$ was arbitrary, it follows that $\G_T$ determines the multiset $M^{(a)}$ for each $a \ge 2$.

    We now show how the leaf adjacency sequence of $T$ can be recovered from the multisets $M^{(a)}$. For each $d \ge 2$, let $V_d$ be the set of vertices of $T$ with degree $d$, so $V^{(a)} = \bigcup_{i \ge 1} V_{1 + i(a-1)}$. Let $M_d$ be the multiset \[
        M_d = \left\{ \frac{\ell(v)}{d-1} : v \in V_d \right\},
    \] so $M^{(a)} = \bigsqcup_{i \ge 1} M_{1 + i(a-1)}$. In particular, if $\Delta$ is the largest degree of a vertex of $T$, then $M^{(\Delta)} = M_{\Delta}$, so the multiset $M_\Delta$ is determined by $\G_T$. Then, for $2 \le d < \Delta$, we have \[
        M_d = M^{(d)} \setminus \bigsqcup_{i \ge 2} M_{1 + i(d-1)},
    \] so by an inductive process, we can recover all the $M_d$ from the $M^{(a)}$ and thus from $\G_T$. Now $M_d$ is clearly equivalent to the multiset of numbers $\ell(v)$ for $v \in V_d$, so the $M_d$ together determine the leaf adjacency sequence of $T$, as desired.
\end{proof}

This result is equivalent to Theorem~\ref{thm:degree-embedding-list}\ref{thm-part:leaf-adjacency}. Indeed, let $\star_{d,k}$ be the star with center vertex labeled $d$ and $k$ outer vertices all labeled $1$. Then we have \[
    N_{\star_{d,k}}(T) = \sum_{\ell \ge k} \binom{\ell}{k} L_{d,\ell}(T),
\] so by triangularity, the numbers $(N_{\star_{d,k}}(T))$ and $(L_{d,\ell}(T))$ determine each other.
(Note that $N_{\star_{d,2}}(T) = N_{1,d,1}(T)$, so this recovers another sequence of triple-degree numbers of $T$.)

\section{Higher-order generalized degree polynomials} \label{sec:higher-order}

In this final section, we discuss a possible further generalization of the generalized degree polynomial: the \emph{generalized degree polynomial of order $m$} for positive integers $m$.

To define this polynomial, we first extend the notion of \emph{boundary edge} from Section~\ref{subsec:GDP} as follows. Let $F$ be a forest and let $A_1, \dots, A_m$ be disjoint subsets of $V(F)$. A \emph{boundary edge} of $(A_1, \dots, A_m)$ is an edge of $F$ that is a boundary edge of at least one $A_i$, that is, the set of boundary edges of $(A_1, \dots, A_m)$ is simply \[
    D(A_1, \dots, A_m) = D(A_1) \cup \dotsm \cup D(A_m).
\] We write $d(A_1, \dots, A_m) = | D(A_1, \dots, A_m)|$. Thus, $d(A_1, \dots, A_m) \le d(A_1) + \dots + d(A_m)$.

\begin{defn}
    The \emph{generalized degree polynomial of order $m$} of a forest $F$ is \[
        \G^{(m)}_F(x_1, \dots, x_m, y, z_1, \dots, z_m)
        = \sum_{A_1 \sqcup \dotsm \sqcup A_m \subseteq V(F)}
        x_1^{|A_1|} \dotsm x_m^{|A_m|}
        \cdot
        y^{d(A_1, \dots, A_m)}
        \cdot
        z_1^{e(A_1)} \dotsm z_m^{e(A_m)},
    \]
where the sum is over tuples $(A_1, \dots, A_m)$ of disjoint subsets of $V(F)$. 
\end{defn}
Note that $\G^{(1)}_F$ is equivalent to the usual generalized degree polynomial $\G_F$. As in Section~\ref{subsec:GDP}, we will show that $\G^{(m)}_F$ is a linear function of $\X_F$. To do this, we first consider a variant of $\X_F$ itself, the \emph{bad chromatic symmetric function} $\B_F$, and show that in fact it is equivalent to $\X_F$. We will then show how $\B_F$ naturally gives rise to $\G^{(m)}_F$.

\begin{defn}
    For a forest $F$, the \emph{bad chromatic symmetric function} of $F$ is
    \[\B_F(x_1, x_2, \dots; z_1, z_2, \dots) = \sum_\kappa \prod_i x_i^{|\kappa^{-1}(i)|} z_i^{b_{\kappa}(i)},\]
    where $\kappa \colon V(F) \to \Zpos$ is any vertex coloring (not necessarily proper) of $F$, and $b_{\kappa}(i)$ is the number of edges $\{v,w\} \in E(F)$ such that $\kappa(v)=\kappa(w)=i$.
\end{defn}
(This can be seen as a further generalization of Stanley's symmetric function generalization of the bad coloring polynomial by setting $z_i = 1+t$; see \cite[Section 3.1]{stanley2}.)

If we let $A_i$ denote the set of vertices colored $i$, then we can alternatively write
\[\B_F(x_1, x_2, \dots; z_1, z_2, \dots) = \sum_{A_1 \sqcup A_2 \sqcup \cdots = V(F)} \prod_i x_i^{|A_i|} z_i^{e(A_i)},\]
where the sum ranges over all ways to express $V(F)$ as a disjoint union of (possibly empty) sets $A_i$.
We denote by $\B^{(m)}_F$ the polynomial in variables $x_1, \dots, x_m, z_1, \dots, z_m$ obtained by setting $x_i = z_i = 0$ in $\B_F$ for $i > m$. 
Observe that $\B_F$ can be recovered as an inverse limit over the polynomials $\B^{(m)}_F$ because any particular monomial of $\B_F$ appears in $\B^{(m)}_F$ for all sufficiently large $m$.

As a formal power series, $\B_F$ is a \emph{MacMahon symmetric function}, as defined in \cite{ROSAS2001326}: that is, it is invariant under the diagonal action of the symmetric group of the positive integers simultaneously permuting the $x$-variables and the $z$-variables. In addition, setting all $z_i = 0$ in $\B_F$ recovers the chromatic symmetric function $\X_F$. Note also that if $A_i$ is a nonempty subset of $V(F)$ inducing a connected subgraph of $F$, then $e(A_i) = |A_i|-1$. Thus, for any partition $\l \vdash |F|$, the coefficient of $\prod_i x_i^{\l_i} z_i^{\l_i-1}$ in $\B_F$ is (up to a factor corresponding to permuting equal parts of $\l$) exactly $b_\l(F)$, the number of connected partitions of $F$ of type $\l$. It follows that one can also recover the expansion of $\X_F$ in the power sum basis from $\B_F$.

While it might seem that $\B_F$ contains strictly more information than $\X_F$, the following proposition shows that $\X_F$ determines $\B^{(m)}_F$ for each $m$, so $\B_F$ and $\X_F$ are in fact equivalent.

\begin{proposition} \label{prop:B-from-X}
    For each $m$, there exists a $\Q$-algebra map $\beta^{(m)}$ such that $\beta^{(m)}(\X_F) = \B_F^{(m)}$ for all forests $F$.
\end{proposition}
\begin{proof}
    Recall the functions $\varphi_{t,u}$ from Lemma~\ref{lemma:phi}. We consider the $m$-fold convolution
    \[\beta^{(m)} = \varphi_{x_1, z_1} * \cdots * \varphi_{x_m, z_m}.\]
    For a forest $F$, iterating \eqref{eq:CSF-comult} yields \[
    \Delta^{m-1}(\X_F) = \sum_{A_1 \sqcup \dotsm \sqcup A_m = V(F)} \X_{F|A_1} \tensor \dotsm \tensor \X_{F|A_m},\]
    where $A_1, \dots, A_m$ partition the vertices of $F$. Therefore, we have
    \[\begin{aligned}
    \beta^{(m)}(\X_F)
    &= \sum_{A_1 \sqcup \dotsm \sqcup A_m = V(F)}
        \varphi_{x_1, z_1}(\X_{F|A_1})
        \, \dotsm \,
        \varphi_{x_m, z_m}(\X_{F|A_m})
    \\ &= \sum_{A_1 \sqcup \dotsm \sqcup A_m = V(F)}
        x_1^{|A_1|} \dotsm x_m^{|A_m|}
        \cdot
        z_1^{e(A_1)} \dotsm z_m^{e(A_m)},
\end{aligned}\]
which is precisely $\B_F^{(m)}$.
\end{proof}

We now show that $\G^{(m)}_F$ is essentially equivalent to $\B_F^{(m+1)}$. Note that if $A_1 \sqcup \cdots \sqcup A_{m+1} = V(F)$, then the boundary edges of $(A_1, \dots, A_m)$ are precisely the edges whose endpoints lie in different subsets $A_i$ and $A_j$ for some $1 \le i \neq j \le m+1$. Hence 
\[d(A_1, \dots, A_m) +  \sum_{i=1}^{m+1} e(A_i) = \sum_{i=1}^{m+1} |A_i| - c(F)\] since both sides count the total number of edges in $F$. Therefore
\begin{align*}
    & \B_F^{(m+1)}(x_1y, \dots, x_my, y; \, y^{-1}z_1, \dots, y^{-1}z_m, y^{-1})
    \\ &= \sum_{A_1 \sqcup \dots \sqcup A_m \sqcup A_{m+1} = V(F)} x_1^{|A_1|} \dotsm x_m^{|A_m|} \cdot y^{\sum_{i=1}^{m+1} (|A_i| - e(A_i))} \cdot z_1^{e(A_1)} \dotsm z_m^{e(A_m)}
    \\ &= \sum_{A_1 \sqcup \dots \sqcup A_m \subseteq V(F)} x_1^{|A_1|} \dotsm x_m^{|A_m|} \cdot y^{d(A_1, \dots, A_m) + c(F)} \cdot z_1^{e(A_1)} \dotsm z_m^{e(A_m)}
    \\ &= y^{c(F)} \G^{(m)}_F(x_1, \dots, x_m, y, z_1, \dots, z_m).
\end{align*} 
It follows that there is a $\Q$-algebra map $\gamma^{(m)}$ that sends $\X_F$ to $y^{c(F)} \G^{(m)}_F$. Specifically, the above computations show that we can take \[
    \gamma^{(m)} = \varphi_{x_1y, y^{-1}z_1} * \dotsm * \varphi_{x_my, y^{-1}z_m} * \varphi_{y, y^{-1}}.
\]
(When $m = 1$, this reduces to the map from Proposition~\ref{prop:GT-linear}.) We note that $\gamma^{(m)}$ is multiplicative and satisfies \[
    \gamma^{(m)}(p_n) = x_1^n y (y-z_1)^{n-1} + \dotsm + x_m^n y (y-z_m)^{n-1} + y (y-1)^{n-1}.
\]

\subsection{Future work}

A natural direction for future work is to investigate what information about a tree $T$ is determined by $\G^{(m)}_T$ for each positive integer $m$. In particular, we pose the following question.

\begin{question} \label{q:same-GTm}
    Does there exist a positive integer $m$ such that $\G^{(m)}_T$ distinguishes trees?
\end{question}

As we already know, it is not true that $\G^{(1)}_T$ (i.e., $\G_T$) distinguishes trees; the smallest pair of non-isomorphic trees with the same generalized degree polynomial is depicted in Figure~\ref{fig:same-gdp}. However, we have not been able to find a pair of non-isomorphic trees $T_1$ and $T_2$ with $\G^{(2)}_{T_1} = \G^{(2)}_{T_2}$. A computer search indicates that no such trees exist with $|T_1| = |T_2| \le 23$.

A positive answer to Question~\ref{q:same-GTm} would immediately imply a positive answer to Question~\ref{q:distinguish}. However, the converse does not necessarily hold, that is, it is possible that the chromatic symmetric function distinguishes trees but no single $\G^{(m)}_T$ does. (This is because our method to recover $\X_T$ from $\B_T$ and thus from the $\G^{(m)}_T$ requires $m = |T|-1$, which is unbounded.) If this is indeed the case, then a construction to produce pairs (or larger sets) of non-isomorphic trees with the same $\G^{(m)}_T$ could be valuable in studying Question~\ref{q:distinguish}.

\section{Acknowledgments}

The authors would like to thank Jeremy Martin for helpful discussions and for welcoming the second author to the University of Kansas in November 2025. They also acknowledge the valuable suggestions of the anonymous referees.

The first author was partially supported by a Simons Foundation Travel Support for Mathematicians Award.

\printbibliography

\end{document}